\documentclass{amsproc}
\usepackage{amsmath}
\usepackage{amsfonts}
\usepackage{amssymb}
\usepackage{hyperref}
\hypersetup{colorlinks,
citecolor=black,
filecolor=black,
linkcolor=black,
urlcolor=black}
\usepackage{graphics,epstopdf}
\usepackage[pdftex]{graphicx}
\usepackage{newlfont}\newlength{\defbaselineskip}
\usepackage[left=1in,right=1in,top=1in,bottom=0.8in,footskip=0.25in]{geometry}
\usepackage{setspace}
\allowdisplaybreaks
\newtheorem{theorem}{Theorem}[section]
\newtheorem{example}{Example}[section]
\newtheorem{lemma}{Lemma}[section]
\newtheorem{definition}[theorem]{Definition}
\newtheorem{remark}{Remark}[section]

\numberwithin{equation}{section}

\usepackage{cite}
\usepackage{fancyhdr}
\usepackage{graphics}
\usepackage{color}
\usepackage{rotating}
\usepackage{fancybox}
\newcommand\norm[1]{\left\lVert#1\right\rVert}
\begin{document}
\title{Convergence of the summation-integral type operators via statistically
}
\maketitle
\begin{center}
{\bf Rishikesh Yadav$^{1,\dag}$,  Ramakanta Meher$^{1,\star}$,  Vishnu Narayan Mishra$^{2,\circledast}$}\\
$^{1}$Applied Mathematics and Humanities Department,
Sardar Vallabhbhai National Institute of Technology Surat, Surat-395 007 (Gujarat), India.\\
$^{2}$Department of Mathematics, Indira Gandhi National Tribal University, Lalpur, Amarkantak-484 887, Anuppur, Madhya Pradesh, India\\
\end{center}
\begin{center}
$^\dag$rishikesh2506@gmail.com,  $^\star$meher\_ramakanta@yahoo.com,
 $^\circledast$vishnunarayanmishra@gmail.com
\end{center}

\vskip0.5in

\begin{abstract}

Our main aim is to investigate the approximation properties for the summation integral type operators in statistical sense. In this regard, we prove the statistical convergence theorem using well known
Korovkin theorem and the degree of approximation is determined. Also using weight function, the weighted statistical convergence theorem with the help of Korovkin theorem is obtained. The statistical rate of convergence in the terms of modulus of continuity and function belonging to the Lipschitz class are obtained. To support the convergence results of the proposed operators to the function, 
 graphical representations take place and a comparison is shown with Sz\'asz-Mirakjan-Kantorovich operators through examples. The last section deals with, a bivariate extension of the proposed operators to study the rate of convergence for the function of two variables, additionally, convergence of the bivariate operators is shown with graphically.

\end{abstract}
\subjclass \textbf{MSC 2010}: {41A25, 41A35, 41A36}.



\section{Introduction}

First of all, Fast \cite{HF} introduced statistical convergence and further investigated by   Steinhaus \cite{HS}, in that order  Schoenberg \cite{IJ} reintroduced as well as gave some basics properties and studied the summability theory of the  statistical convergence.

Nowadays, statistical convergence has become an area, which is broad and also very active, even though it has been introduced over fifty years ago and this is being used very frequently in many areas, we refer to some citations as \cite{OD,AZ,IJ1,JC,JC1,JC2}.
 Also, this area is being concerned to investigate the approximation properties of quantum calculus. For instance, statistical approximation properties are used for studying in various papers \cite{VG1,MO,AA,SE,HA,NM,CR}.
 These mentioned operators were investigated via statistical convergence  as well.

In 2003, Duman \cite{OD1}, studied the  $A$-statistical convergence of the linear positive operators for function belonging to the space of all $2\pi$-periodic and continuous functions on whole real line, where $A$ represents non-negative regular summability matrix. 

Since the statistical type of convergence has not been examined so far in the theory of approximation. So the main objective of this paper is to investigate the theory of approximation by considering the  statistical convergence. Here the Korovkin theory is considered to deals with the approximation of function $f$ by the operators   $\{\tilde{S}_{n,a}^*(f;x)\}$ \cite{RY1}  for the summation integral type operators, which are as follows.




\begin{eqnarray}\label{NO}
\tilde{S}_{n,a}^*(f;x)=n\sum\limits_{k=0}^{\infty}s_n^a(x)\int\limits_{\frac{k}{n}}^{\frac{k+1}{n}}f(u)\,du,~~f\in C[0,\infty),~\forall~ x\in[0,\infty), n\in\mathbb{N}.
\end{eqnarray}
where, $f\in C[0,\infty)$, for all $[0,\infty)$ and $n\in\mathbb{N}$,
and studied the Sz$\acute{\text{a}}$sz-Mirakjan- Kantorovich type operators along with its rate of convergence in the sense of local approximation results with the help of modulus of smoothness, second order modulus of continuity, Peetres K-functional and functions belonging to the Lipschitz class. Further, for computing the order of approximation of the operators, we discuss the weighted approximation properties by using the weighted modulus of continuity and prove the theorem.\\
Here it is considered some Lemmas those will be useful  in the further study of theorem. Consider the function $e_i=x^i$, $i=0\cdots4$, we yield a lemma.

\begin{lemma}\cite{RY1}\label{L1}\em{
For each $x\in[0,\infty)$ and $a>1$ fixed, we have
\begin{eqnarray*}
&& 1.~\tilde{S}_{n,a}^*(e_0;x)=1,\\
&& 2.~\tilde{S}_{n,a}^*(e_1;x)=\frac{1}{2n}+\frac{x\log{a}}{\left(-1+a^{\frac{1}{n}}\right)n},\\ 
&& 3.~\tilde{S}_{n,a}^*(e_2;x)=\frac{1}{3n^2}+\frac{2x\log{a}}{\left(-1+a^{\frac{1}{n}}\right)n^2}+\frac{x^2(\log{a})^2}{\left(-1+a^{\frac{1}{n}}\right)^2n^2},\\
&& 4.~\tilde{S}_{n,a}^*(e_3;x)=\frac{1}{4n^3}+\frac{7}{2}\frac{x\log{a}}{\left(-1+a^{\frac{1}{n}}\right)n^3}+\frac{9}{2}\frac{x^2(\log{a})^2}{\left(-1+a^{\frac{1}{n}}\right)^2n^3}+\frac{x^3(\log{a})^3}{\left(-1+a^{\frac{1}{n}}\right)^3n^3}.
\end{eqnarray*}}
\end{lemma}
To find central moments for the defined operators, we have   
 \begin{lemma}\cite{RY1}\label{L2}\em{
For each $x\geq0$, we have
\begin{eqnarray*}
&&1.~\tilde{S}_{n,a}^*(\xi_x(u);x)=-\frac{(-1+2nx)}{2n}+\frac{x\log{a}}{n(-1+a^{\frac{1}{n}})},\\
&&2.~ \tilde{S}_{n,a}^*(\xi_x^2(u);x)=\frac{(1-3nx+3n^2x^2)}{3n^2}-\frac{2(-1+a^{\frac{1}{n}})(-1+nx)x\log{a}}{\left(-1+a^{\frac{1}{n}}\right)^2n^2}+\frac{x^2(\log{a})^2}{\left(-1+a^{\frac{1}{n}}\right)^2n^2},\\
&&3.~\tilde{S}_{n,a}^*(\xi_x^3(u);x) = -\frac{(-1+4nx-6n^2x^2+4n^3x^3)}{4n^3}+\frac{x(7-12nx+6n^2x^2)\log{a}}{2\left(-1+a^{\frac{1}{n}}\right) n^3} \\&&\hspace{3 cm}-\frac{3x^2(-3+2nx)(\log{a})^2+4x^3(\log{a})^3}{2\left(-1+a^{\frac{1}{n}}\right)^2n^3},\\
&&4.~\tilde{S}_{n,a}^*(\xi_x^4(u);x) = \frac{1}{5\left(-1+a^{\frac{1}{n}} \right)^4n^4}\Bigg(\left(-1+a^{\frac{1}{n}}\right)^4(1-5nx+10n^2x^2-10n^3x^3+5n^4x^4) \\&&\hspace{3 cm}-10\left(-1+a^{\frac{1}{n}}\right)^3 x(-3+7nx-6n^2x^2+2n^3x^3)\log{a}\\&&\hspace{3 cm}+15\left(-1+a^{\frac{1}{n}}\right)^2 x^2 (5-6nx+2n^2x^2)(\log{a})^2\\&&\hspace{3 cm}-20\left(-1+a^{\frac{1}{n}}\right) x^3(-2+nx)(\log{a})^3 +5x^4(\log{a})^4\Bigg).
\end{eqnarray*}}
where,  $\xi_x(t)=(u-x)^i$, $i=1,2,3\cdots$
\end{lemma}

\section{Korovkin and Weierstrass type statistical theorem}
If $\{\mathcal{O}_n(f;x)\}$ is a sequence of linear positive  operators such that the sequence $\{\mathcal{O}_n(1;x)\}$,  $\{\mathcal{O}_n(t;x)\}$, $\{\mathcal{O}_n(t^2;x)\}$ converge uniformly  to $1,~x,~x^2$ respectively in the defined  interval $[a,b]$, then it  implies that the sequence   $\{\mathcal{O}_n(f;x)\}$  converges to the function $f$ uniformly  provided $f$ is bounded and continuous in the interval $[a,b]$.

 Before proceeding to statistical convergence, here  a brief concept of statistical convergence is considered.

\begin{definition}
Consider a set $\mathcal{K}\subseteq \mathbb{N}$, such that $K_n=\{k\in \mathcal{K}: k\leq n \}$, where $n\in\mathbb{N}$.
Then the natural density $d(\mathcal{K})$ of a set $\mathcal{K}$ is defined as
\begin{eqnarray}
\underset{n\to\infty}\lim\frac{1}{n}|K_n|,
\end{eqnarray} 
provided the limit exists. Here $|K_n|$ represents the cardinality of the set $K_n$.
\end{definition}

\begin{definition}
Let $p\in\mathbb{R}$, a sequence $\{x_n\}$ is said to be convergent statistically to $p$, if for each $\epsilon>0$, we have
\begin{eqnarray}
d(\{n\leq k: |x_n-p| \geq\epsilon\})=0
\end{eqnarray}
i.e.
\begin{eqnarray}
\underset{n}\lim\frac{1}{n}|\{n\leq k: |x_n-p| \geq\epsilon\}|=0,~~~~or,~~ |x_n-p|<\epsilon,~~~~\text{for almost~all}~k.
\end{eqnarray}
\end{definition}

By using the properties of statistical convergence, here we shall prove the Korovkin theorem and the Weierstrass type approximation theorem.

In  \cite{PP}, it is proved that the classical Korovkin theorem, according to that theorem, let $A_n(f;x)$ be linear positive operators  
defined on the set of all continuous and bounded function $C_B[a,b]$ to $C[a,b]$, be set of all continuous function defined on $[a,b]$ for which, the condition \label{T1}
\begin{eqnarray*}
 \underset{n}\lim\|A_n(e_i;x)-e_i\|_{C[a,b]}&=&0,~~~~~~~~~~~~\text{where}~e_i=x^i,~~i=0,~1,~2,
\end{eqnarray*}
satisfy, then for any function $f\in C[a,b]$, 
\begin{eqnarray*}
\underset{n}\lim\|S_n(f;x)-f(x)\|_{C[a,b]}=0,~~~\text{as}~n\to\infty.
\end{eqnarray*}
\begin{theorem}\cite{AD}\label{T2}
Let $P_n$ be a sequence of positive linear operators defined on $C_B[a,b]$ to $C[a,b]$ and if it satisfies the conditions  
 \begin{eqnarray*}
 st-\underset{n}\lim\|P_n(e_i;x)-e_i\|_{C[a,b]}&=&0,\hspace{2cm} \text{where}~~ i=0,1,2,
\end{eqnarray*}
then for each function $f\in C_\mathcal{B}[a,b]$, we have   
\begin{eqnarray*}
st- \underset{n}\lim\|P_n(f;x)-f(x)\|_{C[a,b]}=0. 
\end{eqnarray*}
\end{theorem}

With these correlation, we have
\begin{theorem}\label{T3}
Let $\{\tilde{S}_{n,a}^*\}$ be the sequence defined by (\ref{NO}), then for every $f\in C_B[0,l],~~~l>0$ such that
$$st- \underset{n}\lim\|\tilde{S}_{n,a}^*(f;x)-f(x)\|=0, $$
where $C_B[0,l]$ is the space of all continuous and bounded function defined on $[0,l]$ with the norm 

$$\Vert f \Vert=\underset{0\leq x\leq l}\sup|f(x)|$$
\end{theorem}

\begin{proof}
By (1) of Lemma \ref{L1}, we easily get
\begin{eqnarray}\label{E1}
st-\underset{n}\lim\|\tilde{S}_{n,a}^*(e_0;x)-e_0\|=0
\end{eqnarray}
Now by (2) of Lemma (\ref{L1}), we have 
\begin{eqnarray*}
|S_n(e_1;x)-x\| &=& \norm{\frac{1}{2n}+\frac{x}{\left(-1+a^{\frac{1}{n}}\right)n}\log{a}-x} \\
&\leq & \left|\frac{1}{2n}\right|+\left|\left\{\frac{1}{\left(-1+a^{\frac{1}{n}}\right)n}\log{a}-1\right\}l\right|,
\end{eqnarray*}

define the sets, for any $\epsilon>0$ as:

$$O=\{n:\|\tilde{S}_{n,a}^*(e_1;x)-x\| \geq \epsilon\}$$ 
and
\begin{eqnarray*}
O' &=& \bigg\{n:\frac{1}{2n}\geq \frac{\epsilon}{2}\bigg\}\\
O'' &=& \bigg\{n: \left(\frac{\log{a}}{\left(-1+a^{\frac{1}{n}}\right)n}-1\right)l\geq \frac{\epsilon}{2 }\bigg\},
\end{eqnarray*}
and $O\subseteq O'\cup O''$ and it can be expressed as 
\begin{scriptsize}
\begin{eqnarray}\label{A}
\nonumber d\left\{n\leq k:\|\tilde{S}_{n,a}^*(e_1;x)-x\|\geq \epsilon\ \right\}&\leq & d\left\{n\leq k:\frac{1}{2n}\geq\frac{\epsilon}{2}\ \right\}\\&& + d\left(n\leq k: \left(\frac{\log{a}}{\left(-1+a^{\frac{1}{n}}\right)n}-1\right)l\geq \frac{\epsilon}{2 } \right)
\end{eqnarray}
\end{scriptsize}
But since

\begin{eqnarray*}
st-\underset{n}\lim \left(\frac{1}{2n}\right)=0~~~~~~~~~~\text{and}~~st-\underset{n}\lim \left(\frac{\log{a}}{\left(-1+a^{\frac{1}{n}}\right)n}-1\right)=0,
\end{eqnarray*}
hence, by Inequality \ref{A}, it follows
\begin{eqnarray}\label{E2}
st-\underset{n}\lim\|S_n(e_1;x)-x\|=0.
\end{eqnarray}
Similarly,

\begin{eqnarray*}
\|\tilde{S}_{n,a}^*(e_2;x)-e_2\|&=& \norm{\frac{1}{3n^2} + \frac{2x\log{a}}{\left(-1+a^{\frac{1}{n}}\right)n^2} + \frac{x^2(\log{a})^2}{\left(-1+a^{\frac{1}{n}}\right)^2n^2}-x^2} \\
&\leq & \left|\frac{1}{3n^2}\right|+\left|\frac{2\log{a}}{\left(-1+a^{\frac{1}{n}}\right)n^2}\right|~l + \left|\left( \frac{(\log{a})^2}{\left(-1+a^{\frac{1}{n}}\right)^2n^2}-1\right)\right|~l^2\\ 
&\leq & m^2\left(    \frac{1}{3n^2}+\frac{2\log{a}}{\left(-1+a^{\frac{1}{n}}\right)n^2} + \Big( \frac{(\log{a})^2}{\left(-1+a^{\frac{1}{n}}\right)^2n^2}-1\Big)  \right),
\end{eqnarray*}
where $m^2=\max\{1,~l,~l^2\}$, 
i.e. 
\begin{eqnarray}
\|\tilde{S}_{n,a}^*(e_2;x)-e_2\|\leq m^2\left\{    \frac{1}{3n^2}+\frac{2\log{a}}{\left(-1+a^{\frac{1}{n}}\right)n^2} + \left\{ \frac{1}{\left(-1+a^{\frac{1}{n}}\right)^2n^2}(\log{a})^2-1\right\}  \right\}.
\end{eqnarray}
So again by defining the following sets and for any $\epsilon>0$, one can find
\begin{eqnarray}
P &=&\{n: \|\tilde{S}_{n,a}^*(e_2;x)-e_2\|\geq\epsilon \}\\
P_1 &=& \{n: \frac{1}{3n^2}\geq \frac{\epsilon}{3m^2} \},\\
P_2 &=& \{n: \frac{2\log{a}}{\left(-1+a^{\frac{1}{n}}\right)n^2}\geq \frac{\epsilon}{3m^2} \}\\
P_3 &=& \left\{n: \left\{ \left(\frac{\log{a}}{\left(-1+a^{\frac{1}{n}}\right)n} \right)^2-1 \right\} \geq\frac{\epsilon}{3m^2} \right\}.
\end{eqnarray} 
where $P\subseteq P_1\cup P_2\cup P_3$, it gives

\begin{eqnarray}\label{IN1}
&&\nonumber d \{n\leq k: \|\tilde{S}_{n,a}^*(e_2;x)-e_2\|\geq\epsilon \}\leq d\left\{n\leq k: \frac{1}{3n^2}\geq \frac{\epsilon}{3m^2} \right\}\\&&\hspace{2cm}+d \left\{n\leq k: \frac{2\log{a}}{\left(-1+a^{\frac{1}{n}}\right)n^2}\geq \frac{\epsilon}{3m^2}\right \}\nonumber\\&&\hspace{2cm} +d \left\{n\leq k: \left( \left(\frac{\log{a}}{\left(-1+a^{\frac{1}{n}}\right)n} \right)^2-1 \right) \geq\frac{\epsilon}{3m^2} \right\}
\end{eqnarray}
Hence
\begin{eqnarray}\label{IN2}
st-\lim{\alpha_n}=0=st-\lim{\beta_n}=st-\lim{\gamma_n},
\end{eqnarray}
where,
\begin{eqnarray*}
\alpha_n=\frac{1}{3n^2},~~\beta_n=\frac{2\log{a}}{\left(-1+a^{\frac{1}{n}}\right)n^2},~~\gamma_n=\left( \left(\frac{\log{a}}{\left(-1+a^{\frac{1}{n}}\right)n} \right)^2-1 \right).
\end{eqnarray*}
So by \ref{IN1} and \ref{IN2}, we have
\begin{eqnarray}\label{E3}
st-\lim\|\tilde{S}_{n,a}^*(e_2;x)-e_2\|=0
\end{eqnarray} 
Hence proved.
\end{proof}


Now there is  an example which satisfy Theorem \ref{T3}, but not the classical Korovkin theorem. 
\begin{example}
Consider a sequence of  linear positive operators $T_n(f;x)$ which are defined on $C_B[0,l]$ by $T_n(f;x)=(1+u_n)\tilde{S}_{n,a}^*$, where  $\tilde{S}_{n,a}^*$ be the sequence positive linear operators and $u_n$ is unbounded statistically convergent sequence. \\

Since $\tilde{S}_{n,a}^*$ is statistically convergent and  also $u_n$ is statistically convergent but not convergent so one can observe that the  sequence $T_n$ satisfies the Theorem \ref{T3}, but not the classical Korovkin theorem.
\end{example}

\begin{definition}
Let $\xi_n$ be a sequence that is converges statistically  to $\xi$, having degree $\beta\in (0,1)$, if for each $\epsilon>0$, we have
\begin{eqnarray*}
\underset{n}\lim\frac{\{n\leq k:|\xi_n-\xi|\geq \epsilon\}}{k^{1-\beta}}=0
\end{eqnarray*} 
In this case, we can write
\begin{eqnarray*}
\xi_n-\xi=st-o(k^{-\beta}),~~~~~~k\to\infty.
\end{eqnarray*}
\end{definition} 

\begin{theorem}
Let $\{\tilde{S}_{n,a}^*\}$ be a sequence defined by (\ref{NO}) that satisfy the conditions 
\begin{eqnarray}
 st-\underset{n\to\infty}\lim\|\tilde{S}_{n,a}^*(e_0;x)-e_0\|&=&st-o(n^{-\zeta_1}),\\
 st-\underset{n\to\infty}\lim\|\tilde{S}_{n,a}^*(e_1;x)-e_1\|&=&st-o(n^{-\zeta_2}),\\
 st-\underset{n\to\infty}\lim\|\tilde{S}_{n,a}^*(e_2;x)-e_2\|&=&st-o(n^{-\zeta_3}),
\end{eqnarray} 
as $n\to\infty$. Then for each $f\in C_B[0,l]$, we have
\begin{eqnarray*}
st-\underset{n\to\infty}\lim\ \tilde{S}_{n,a}^*(f;x)-f(x)\|&=& st-o(n^{-\zeta}),\,~~~~~~~\text{as}~n\to\infty,
\end{eqnarray*}
where  $\zeta=\min\{\zeta_1,~\zeta_2,~\zeta_3 \}$.
\end{theorem}

\begin{proof}
One can write the inequality (\ref{IN1}) of Theorem \ref{T3} as:

\begin{eqnarray*}
&&\frac{| \{n\leq p: \|\tilde{S}_{n,a}^*(e_2;x)-e_2\|\geq\epsilon \}|}{p^{1-\zeta}} \leq \frac{\left| \left\{n\leq p: \frac{1}{3n^2}\geq \frac{\epsilon}{3m^2} \right\}\right|}{p^{1-\zeta_1}} \frac{p^{1-\zeta_1}}{p^{1-\zeta}}\\&&\hspace{5cm}+\frac{\left| \left\{n\leq p: \frac{2\log{a}}{\left(-1+a^{\frac{1}{n}}\right)n^2}\geq \frac{\epsilon}{3m^2} \right\}\right|}{p^{1-\zeta_2}}\frac{p^{1-\zeta_2}}{p^{1-\zeta}}\nonumber\\&&\hspace{5cm} +\frac{\left| \left\{n\leq p: \left( \left(\frac{\log{a}}{\left(-1+a^{\frac{1}{n}}\right)n} \right)^2-1 \right) \geq\frac{e}{3m^2} \right\}\right|}{p^{1-\zeta_3}}\frac{p^{1-\zeta_3}}{p^{1-\zeta}}.
\end{eqnarray*}
By letting $\zeta=\min\{\zeta_1,~\zeta_2,~\zeta_3 \}$ and as $n\to\infty$ the desired result can be achieved. 
\end{proof}



\section{Weighted Statistical Convergence}
Next, we introduce the convergence properties of the given operators (\ref{NO}) using Korovkin type theorem , recall from \cite{GA1,GA2}, here the weight function is $w(x)=1+\gamma^2(x)$, where $\gamma:\mathbb{R_+}\rightarrow\mathbb{R_+}$ is an unbounded strictly increasing continuous function for which, there exist $M>0$ and $\alpha\in(0,1]$ such that: 

\begin{eqnarray*}
|x-y|\leq M| \gamma(x)-\gamma(y)|, \;\;\ \forall\; x,y\geq 0.
\end{eqnarray*}

Consider $w(x)=1+x^2$ be a weighted function and let $B_w[0,\infty)$ be the space defined by:
\begin{center}
$B_w[0,\infty)=\left\{f:[0,\infty)\to \mathbb{R}~~\vert~~ \|f\|_{w}=\sup\limits_
{x\ge 0}\frac{f(x)}{w(x)}< +\infty \right\}.$
\end{center}
Also the spaces 
\begin{eqnarray*}
C_w[0,\infty)&=& \{f\in B_w[0,\infty), ~~f~~ \text{is continuous}\},\\ C_w^k[0,\infty)&=& \left\{f\in C_w[0,\infty),~ \lim\limits_{x\rightarrow\infty}\frac{f(x)}{w(x)}=k_{f}<+\infty \right\}.
\end{eqnarray*}

 We can move towards the main theorem by considering the functions from the above defined spaces via statistically.
 
\begin{theorem}
Let $\tilde{S}_{n,a}^*f$ be a linear positive operators defined by (\ref{NO}), then for each $f\in C_w^k[0,\infty)$, we have 
\begin{center}
$st-\lim\limits_{n\rightarrow\infty}\|\tilde{S}_{n,a}^*(f;x)-f(x)\|_w=0.$
\end{center}
\end{theorem}

\begin{proof}
Using the Lemma \ref{L1}, we have $ \tilde{S}_{n,a}^*(e_0;x)=1$, then it is obvious that
\begin{eqnarray*}
\| \tilde{S}_{n,a}^*(e_0;x)-1\|_w=0.
\end{eqnarray*}
Now we have

\begin{eqnarray*}
\| \tilde{S}_{n,a}^*(e_1;x)-e_1)\|_w &=& \sup\limits_{x\geq 0} \left(\frac{1}{2n}+\left\{ \frac{1}{\left(-1+a^{\frac{1}{n}}\right)n}\log{a}-1\right\} x\right)\frac{1}{1+x^2}\\ &=& \sup\limits_{x\geq 0} \left(\frac{1}{2n}\frac{1}{1+x^2}+\frac{x\log{a}}{\left(-1+a^{\frac{1}{n}}\right)n}\frac{1}{1+x^2}-\frac{x}{1+x^2}\right) \\ &=& \sup\limits_{x\geq 0} \left(\frac{1}{2n}\frac{1}{1+x^2}+ \left( \frac{\log{a}}{\left(-1+a^{\frac{1}{n}}\right)n}-1\right)\frac{x}{1+x^2}\right) \\ &\leq & \left(\frac{1}{2n}+ \left( \frac{\log{a}}{\left(-1+a^{\frac{1}{n}}\right)n}-1\right)\right) 
\end{eqnarray*}

Now for any $\epsilon>0$, on defining the following sets:
\begin{eqnarray*}
P &=& \left\{n:\| \tilde{S}_{n,a}^*(e_1;x)-e_1)\|_w \geq \epsilon \right\}\\
P' &=& \left\{n: \frac{1}{2n}\geq\frac{\epsilon}{2} \right\}\\
P'' &=& \left\{n:  \left( \frac{\log{a}}{\left(-1+a^{\frac{1}{n}}\right)n}-1\right)\geq \frac{\epsilon}{2}\right\}, 
\end{eqnarray*}
where $P\subseteq P' \cup P''$, it follows
\begin{eqnarray}\label{EQ1}
&&d  \left\{n\leq m:\| \tilde{S}_{n,a}^*(e_1;x)-e_1)\|_w \geq \epsilon \right\}  \leq d \left\{n\leq m: \frac{1}{2n}\geq\frac{\epsilon}{2} \right\}\nonumber \\  &&\hspace{6cm}+d \left\{n\leq m:  \left( \frac{\log{a}}{\left(-1+a^{\frac{1}{n}}\right)n}-1\right)\geq \frac{\epsilon}{2}\right\}. 
\end{eqnarray}

Right hand side of the above inequality (\ref{EQ1}) is statistically convergent, hence
\begin{eqnarray}
st-\underset{n}\lim \| \tilde{S}_{n,a}^*(e_1;x)-e_1)\|_w=0
\end{eqnarray}

Similarly,
\begin{eqnarray}\label{EQ2}
\| \tilde{S}_{n,a}^*(e_2;x)-e_2\|_w \nonumber &=& \sup\limits_{x\geq 0}\left\{\frac{1}{3n^2}+\frac{2x\log{a}}{\left\{-1+a^{\frac{1}{n}}\right\}n^2}+\frac{x^2(\log{a})^2}{\left\{-1+a^{\frac{1}{n}}\right\}^2n^2}-x^2\right\}\frac{1}{1+x^2}\\ \nonumber 
& \leq & \left\{\frac{1}{3n^2}+\frac{2\log{a}}{\left(-1+a^{\frac{1}{n}}\right)n^2}+\left(\frac{(\log{a})^2}{\left(-1+a^{\frac{1}{n}}\right)^2n^2}-1\right)\right\}.
\end{eqnarray}
Similarly for any $\epsilon>0$, by defining following define sets
\begin{eqnarray*}
H&=&\left\{n: \| \tilde{S}_{n,a}^*(e_2;x)-e_2\|_w\geq \epsilon \right\}\\
H'&=&\left\{n: \frac{1}{3n^2}\geq \frac{\epsilon}{3} \right\}\\
H''&=&\left\{n: \frac{2\log{a}}{\left(-1+a^{\frac{1}{n}}\right)n^2}\geq\frac{\epsilon}{3} \right\}\\
H'''&=&\left\{n: \left(\frac{(\log{a})^2}{\left(-1+a^{\frac{1}{n}}\right)^2n^2}-1\right)\geq\frac{\epsilon}{3} \right\},
\end{eqnarray*}
where $H\subseteq H'\cup H''\cup H'''$, it follows
\begin{eqnarray}
d\{n\leq m: \| \tilde{S}_{n,a}^*(e_2;x)-e_2\|_w\geq \epsilon \}&=&0\label{E3}\\
d\left\{n\leq m: \frac{2\log{a}}{\left(-1+a^{\frac{1}{n}}\right)n^2}\geq\frac{\epsilon}{3} \right\}&=&0\label{E4}\\
d\left\{n\leq m: \left(\frac{(\log{a})^2}{\left(-1+a^{\frac{1}{n}}\right)^2n^2}-1\right)\geq\frac{\epsilon}{3} \right\}&=&0.\label{E5}
\end{eqnarray}
By the above relations (\ref{E3},\ref{E4},\ref{E5}), we yield as: 
\begin{eqnarray}
st-\lim\limits_{n\rightarrow\infty}\|\tilde{S}_{n,a}^*(e_2;x)-e_2\|_w=0.
\end{eqnarray}
Hence,
\begin{eqnarray}
|\tilde{S}_{n,a}^*(f;x)-f(x)\|_w\leq  \| \tilde{S}_{n,a}^*(e_0;x)-e_0)\|_w +  \| \tilde{S}_{n,a}^*(e_1;x)-e_1)\|_w +  \| \tilde{S}_{n,a}^*(e_2;x)-e_2)\|_w,
\end{eqnarray}
we get,

\begin{eqnarray}
st-\lim\limits_{n\rightarrow\infty} \|\tilde{S}_{n,a}^*(f;x)-f(x)\|_w &\leq &  st-\lim\limits_{n\rightarrow\infty}  \| \tilde{S}_{n,a}^*(e_0;x)-e_0)\|_w  + st-\lim\limits_{n\rightarrow\infty}  \| \tilde{S}_{n,a}^*(e_1;x)-e_1)\|_w \\ \nonumber \\ \hspace{6cm} &+& st-\lim\limits_{n\rightarrow\infty} \| \tilde{S}_{n,a}^*(e_2;x)-e_2)\|_w \nonumber,
\end{eqnarray}
which implies that 
\begin{center}
$st-\lim\limits_{n\rightarrow\infty}\|\tilde{S}_{n,a}^*(f;x)-f(x)\|_w=0.$
\end{center}
Hence proved.
\end{proof}
\section{Rate of Statistical Convergence}
In this section, we shall introduce the order of approximation of the operators by means of the modulus of continuity and function belonging to the Lipschitz class.

Let $f\in C_B[0,\infty)$, the space of all continuous and bounded functions defined on the interval $[0,\infty)$ and for any  $x\geq 0$,  the modulus of continuity of $f$ is defined to be 
\begin{eqnarray*}
\omega(f;\delta)=\underset{|u-x|\leq\delta}\sup |f(u)-f(x)|,~~~~~~~~~~~~~~~u\in[0,\infty).
\end{eqnarray*}  
And for any $\delta>0$ and each $x,~u\in [0,\infty)$, we have
\begin{eqnarray}\label{E6}
|f(u)-f(x)|\leq \omega(f;\delta)\left(\frac{|u-x|}{\delta}+1 \right)
\end{eqnarray}
 
Next theorem deals with error estimation using the modulus of continuity:

\begin{theorem}
Let $f\in C_B[0,\infty)$ be a non-decreasing function then we have
\begin{eqnarray*}
|\tilde{S}_{n,a}^*(f;x)-f(x)|\leq 2\omega\left(f;\sqrt{\delta_{n,a}} \right),~~~~~~~~x\geq 0,
\end{eqnarray*}
where
\begin{eqnarray*}
\delta_{n,a}=\Bigg(\frac{(1-3nx+3n^2x^2)}{3n^2}-\frac{2(-1+a^{\frac{1}{n}})(-1+nx)x\log{a}}{\left(-1+a^{\frac{1}{n}}\right)^2n^2}+\frac{x^2(\log{a})^2}{\left(-1+a^{\frac{1}{n}}\right)^2n^2} \Bigg).
\end{eqnarray*}
\end{theorem}

\begin{proof}
With the linearity and positivity properties of the defined operators (\ref{NO}), it can be expressed as
\begin{eqnarray}
|\tilde{S}_{n,a}^*(f;x)-f(x)| &\leq & \tilde{S}_{n,a}^*(|f(u)-f(x)|;x)\nonumber\\
&=& n\sum\limits_{k=0}^{\infty}s_n^a(x)\int\limits_{\frac{k}{n}}^{\frac{k+1}{n}}|f(u)-f(x)|\,du
\end{eqnarray}
By using the  inequality (\ref{E6}),  the above inequality can be written as:

\begin{eqnarray*}
|\tilde{S}_{n,a}^*(f;x)-f(x)| &\leq & n\sum\limits_{k=0}^{\infty}s_n^a(x)\int\limits_{\frac{k}{n}}^{\frac{k+1}{n}} \omega(f;\delta)\left(\frac{|u-x|}{\delta}+1 \right)~du\\
&=& \omega(f;\delta) \Big\{1+\frac{n}{\delta} \sum\limits_{k=0}^{\infty}s_n^a(x)\int\limits_{\frac{k}{n}}^{\frac{k+1}{n}} |u-x|~du \Big\}\\ 
& \leq & \omega(f;\delta) \Big\{1+\frac{n}{\delta} \Big( \Big( \sum\limits_{k=0}^{\infty} s_n^a(x)\int\limits_{\frac{k}{n}}^{\frac{k+1}{n}} (u-x)^2~du \Big)^{\frac{1}{2}}\\
&=& \omega(f;\delta) \Big\{1+\frac{1}{\delta} \sqrt{\tilde{S}_{n,a}^*(\xi_x^2(u);x)} \Big\} \\
\end{eqnarray*}

By choosing, $\delta=\delta_{n,a}$, where 
\begin{eqnarray}\label{E7}
\delta_{n,a}=\Bigg(\frac{(1-3nx+3n^2x^2)}{3n^2}-\frac{2(-1+a^{\frac{1}{n}})(-1+nx)x\log{a}}{\left(-1+a^{\frac{1}{n}}\right)^2n^2}+\frac{x^2(\log{a})^2}{\left(-1+a^{\frac{1}{n}}\right)^2n^2} \Bigg).
\end{eqnarray}
Hence, the required result can be obtained.
\end{proof}

\begin{remark}
By using the above equation (\ref{E7}), one can find that
\begin{eqnarray}
st-\underset{n}\lim~\delta_{n,a}=0,
\end{eqnarray}
By (\ref{E6}), similarly one can get 
\begin{eqnarray}
st-\underset{n}\lim~\omega(f;\delta_{n,a})=0
\end{eqnarray}
and hence  the pointwise rate of convergence of the operators $\tilde{S}_{n,a}^*(f;x)$ can be obtained.
\end{remark}

\begin{theorem}\cite{RY1}\em{
Let $f\in C_B[0,\infty)$ and if $f\in\text{Lip}_\mathcal{M}(\alpha),~~~\alpha\in(0,1]$ holds that is the inequality 
\begin{eqnarray*}
|f(u)-f(x)|\leq \mathcal{M}|u-x|^{\alpha},~~~~~u,x\in[0,\infty),~\text{where}~~~\mathcal{M} \text{~is a positive constant}
\end{eqnarray*}
then for every $x\geq 0$, we have 
\begin{eqnarray*}
|\tilde{S}_{n,a}^*(f;x)-f(x)|\leq \mathcal{M}\delta_{n,a}^{\frac{\alpha}{2}},
\end{eqnarray*}
where $\delta_{n,a}=\tilde{S}_{n,a}^*((u-x)^{2};x)$.
}
\end{theorem}

\begin{proof}
Since we have $f\in C_B[0,\infty)\cap\text{Lip}_\mathcal{M}(\alpha)$,  so
\begin{eqnarray*}
|\tilde{S}_{n,a}^*(f;x)-f(x)| & \leq & \tilde{S}_{n,a}^*(|f(u)-f(x)|;x)\\&\leq & M \tilde{S}_{n,a}^*(|u-x|^{\alpha};x) =  \mathcal{M} \Big(n\sum\limits_{k=0}^{\infty}s_n^a(x)\int\limits_{\frac{k}{n}}^{\frac{k+1}{n}}|u-x|^{\alpha}~du\Big). 
\end{eqnarray*}
Now by applying H$\ddot{o}$lder inequality with $\mathsf{p}=\frac{2}{\alpha}$ and $\mathsf{q}=\frac{2}{2-\alpha}$, we have 
\begin{eqnarray*}
|\tilde{S}_{n,a}^*(f;x)-f(x)| & \leq & \mathcal{M}\left(n\sum\limits_{k=0}^{\infty}s_n^a(x)\Big\{\int\limits_{\frac{k}{n}}^{\frac{k+1}{n}}(u-x)^2~du,\Big\}^{\frac{\alpha}{2}}\right) \leq \mathcal{M}(\tilde{S}_{n,a}^*(\xi_x^2(u);x))^\frac{\alpha}{2}\\ &=& \mathcal{M}\delta_{n,a}^{\frac{\alpha}{2}}.
\end{eqnarray*}
Hence proved.
\end{proof}

\begin{remark}
Similarly by equation (\ref{E7}), we can justify 
\begin{eqnarray}
st-\underset{n}\lim~\delta_{n,a}=0,
\end{eqnarray}
and it can be seen that the rate of statistical convergence of the operators \ref{NO} to $f(x)$ are estimated by means of function belonging to the Lipschitz class.
\end{remark}
\pagebreak
\section{Graphical approach}
Based upon the defined operator  (\ref{NO}), now we will show the convergence of the operators, for different functions for particular the values of $n$. 

\begin{example}
Let the function $f(x)=e^{-2x}$ and choose the values of $n=5,10$ for which the corresponding operators are $\tilde{S}_{5,a}^*(f;x), \tilde{S}_{10,a}^*(f;x)$ respectively. The approach of the operators is faster to the function as in increment in the value of $n$ (when we choose $n=500, 1000$) and we can observe, the error is able to obtain as small as we please as the the value of $n$ is large, which can be seen in the given figures (\ref{F1}).

\begin{figure}[htb]
    \centering 
    \includegraphics[width=0.4\textwidth]{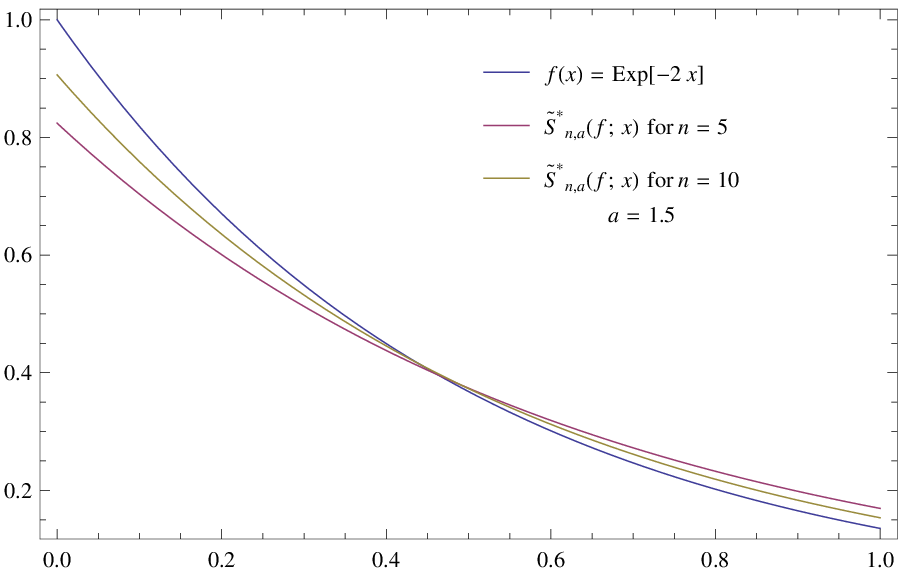}    \includegraphics[width=0.4\textwidth]{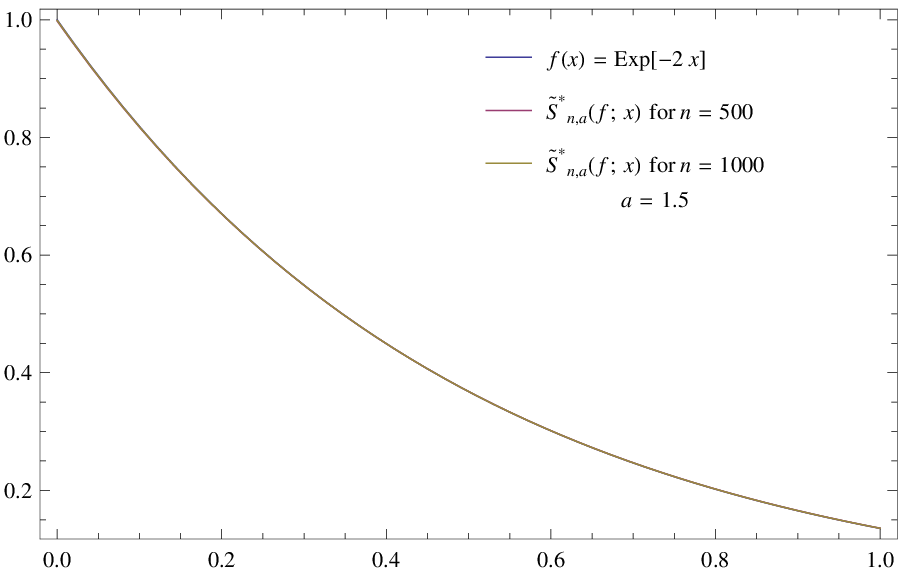} 
    \caption[Description in LOF, taken from~\cite{source}]{Convergence of the operators $\tilde{S}_{n,a}^*(exp(-2u);x)$ to $f(x)$}
    \label{F1}
\end{figure}
\end{example}
In fact, in Figure \ref{F1}, as the value of $n$ increases, the operator $\tilde{S}_{n,a}^*(f;x)$ approaches  towards the function $f(x)=exp(-2x)$ keeping $a=1.5$ fixed.
\begin{example}
Here, we consider the function $f(x)=x$ and choose the value of $n=100, 500$,the convergence can be seen by graphical representation, given by figure \ref{F2}. 
\begin{figure}[htb]
 \centering 
\includegraphics[width=0.4\textwidth]{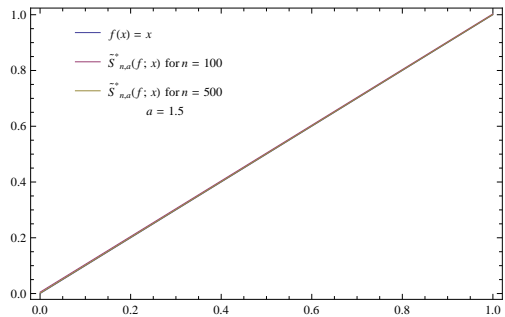} 
  \caption[Description in LOF, taken from~\cite{source}]{Convergence of the operators $\tilde{S}_{n,a}^*(u;x)$ to $f(x)$}
    \label{F2}
\end{figure}
\end{example}

\begin{example}
For the convergence of the proposed  operators (\ref{NO}) to the  function $f(x)=\left(x-\frac{1}{2} \right)\left(x-\frac{1}{3} \right)\left(x-\frac{1}{4} \right)$, choose $n=15, 30, 500, 1000$ and the errors can be observed by the given figures (\ref{F3}).
\begin{figure}[htb]
    \centering 
    \includegraphics[width=0.4\textwidth]{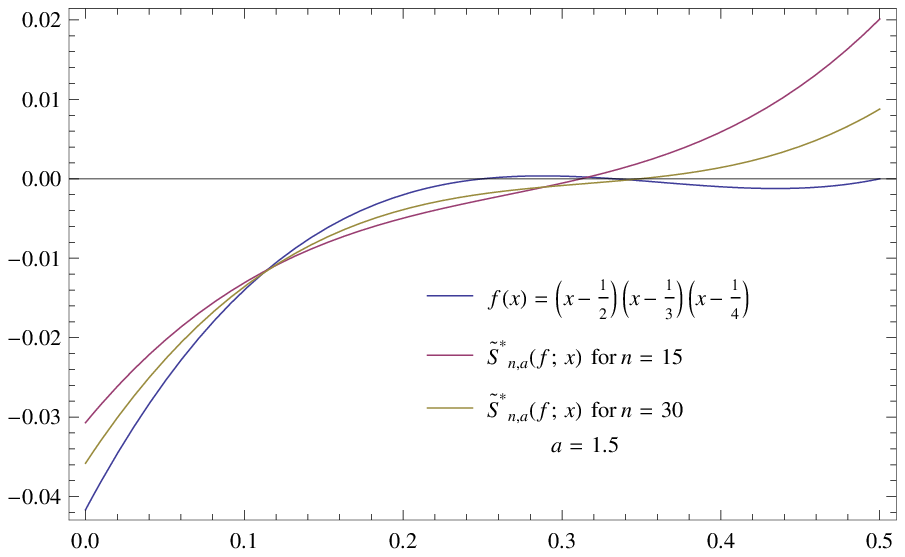}   \includegraphics[width=0.4\textwidth]{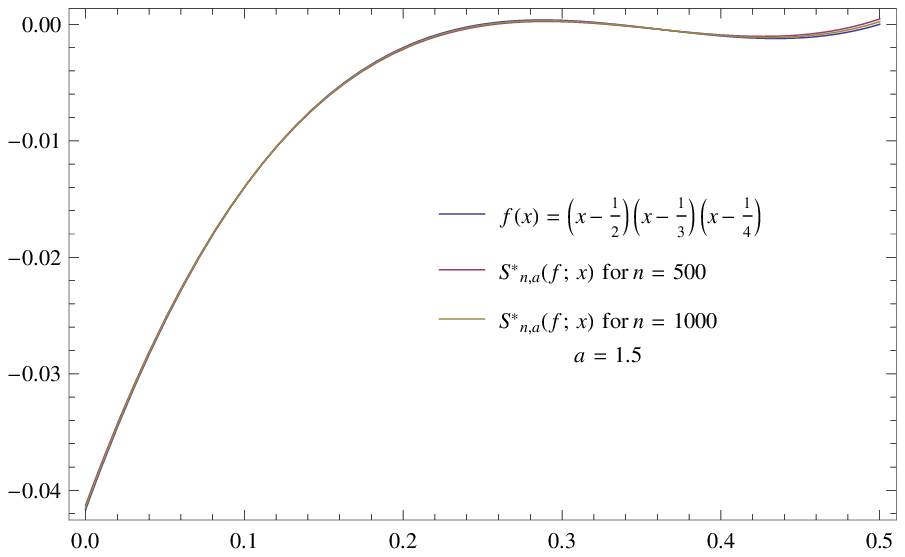} 
      \caption[Description in LOF, taken from~\cite{source}]{Convergence of the operators $\tilde{S}_{n,a}^*(f;x)$ to $f(x)$}
    \label{F3}
\end{figure}
\end{example}
\pagebreak
\textbf{Concluding Remark:}
Since, we have seen the approach of the proposed operators by all the above figures, we can observe that as the degree of the operators is large, the approximation will be better. Moreover, one can observe by Figure \ref{F2}, as the value of $n$ is increased, the operators $\tilde{S}_{n,a}^*(f;x)$  converge to the function $f(x)=x$ and in Figure \ref{F3}, it can be seen that the operator $\tilde{S}_{n,a}^*(t;x)$ converges to the function  $f(x)=\left(x-\frac{1}{2} \right)\left(x-\frac{1}{3} \right)\left(x-\frac{1}{4} \right)$ for large value of $n$.

 \subsection{A comparison to the Sz$\acute{\text{a}}$sz-Mirakjan-Kantorovich operators }

In 1983, V. Totik \cite{VT1} Introduced the Kantorovich variant of the Sz$\acute{\text{a}}$sz-Mirakjan operators in $L^{p}$-spaces for $p>1$, which are as follows:

\begin{eqnarray}\label{KO}
K_n(f;x)=ne^{-nx}\sum\limits_{k=0}^{\infty}\frac{(nx)^k}{k!}\int\limits_{\frac{k}{n}}^{\frac{k+1}{n}}f(u)~du.
\end{eqnarray}
Now we shall show a comparison with the above operators (\ref{KO}) to the proposed operators defined by (\ref{NO}) by graphical representation and convergence behavior will also be seen to the given function.

In Figure \ref{F4}, one can see that the said operators have a better rate of convergence as compared to the above operators (\ref{KO}), but both operators converges to the function $f(x)$.
\begin{example}
For the same degree of approximation of the operators $\tilde{S}_{n,a}^*(f;x)$ and $K_n(f;x)$ to the function $f(x)=x^3$, the comparison shown by Figure \ref{F4}.
\end{example}
\begin{figure}[htb]
    \centering 
    \includegraphics[width=0.4\textwidth]{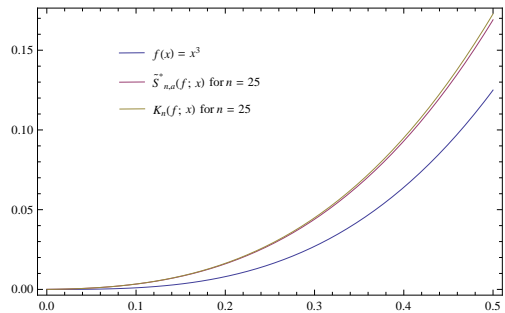}   
      \caption[Description in LOF, taken from~\cite{source}]{The comparison of the operators $\tilde{S}_{n,a}^*(f;x)$ and $K_n(f;x)$}
    \label{F4}
\end{figure}
\pagebreak
\begin{example}
Comparison of convergence can be seen for the operators $\tilde{S}_{n,a}^*(f;x)$ and $K_n(f;x)$ to the function in the given Figures \ref{F5} by choosing $n=4, 20$.
\begin{figure}[htb]
    \centering 
  \includegraphics[width=0.4\textwidth]{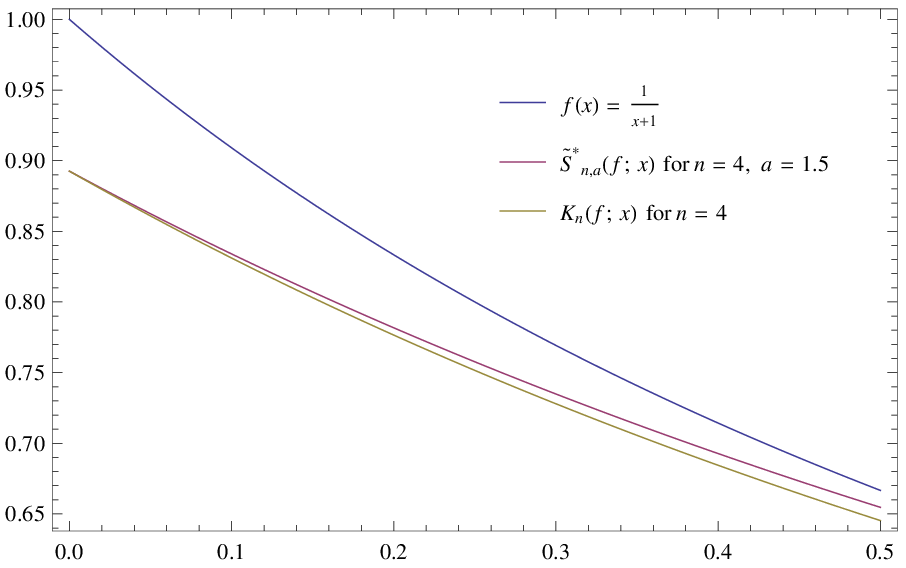}     \includegraphics[width=0.4\textwidth]{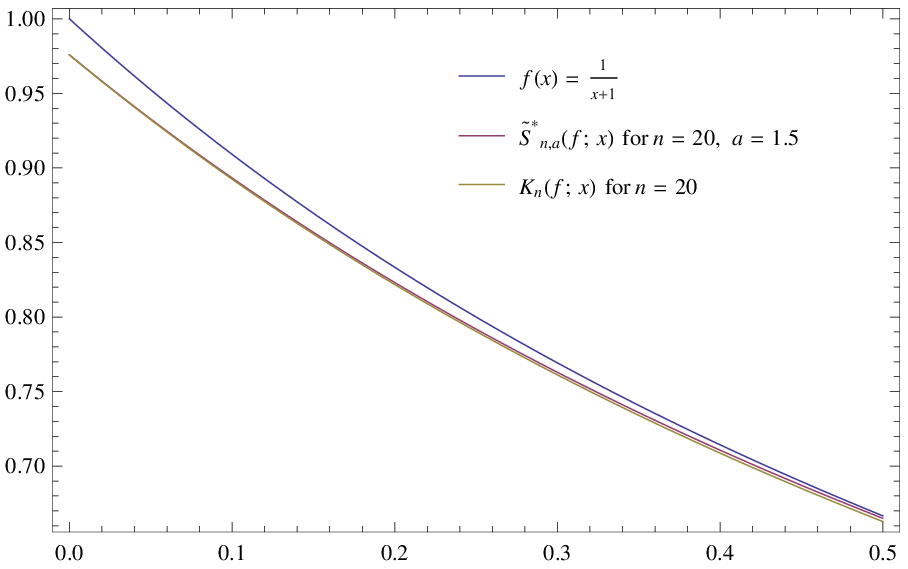}   
      \caption[Description in LOF, taken from~\cite{source}]{The comparison  of the operators $\tilde{S}_{n,a}^*(f;x)$ and $K_n(f;x)$}
    \label{F5}
\end{figure}
\end{example}
So, from Figure \ref{F5}, it can be observed that the summation-integral-type operators (\ref{NO}) are approaching more faster than Sz$\acute{\text{a}}$sz-Mirakjan-Kantorovich operators, but for large value of $n$, both operators converge to the function $f(x)=\frac{1}{1+x}$.
\begin{example}
Consider $n=5$ and function $f(x)=\cos{\pi x}$ then the approaching of the operators $\tilde{S}_{n,a}^*(f;x)$ is faster than $K_n(f;x)$.
\begin{figure}[htb]
    \centering 
  \includegraphics[width=0.5\textwidth]{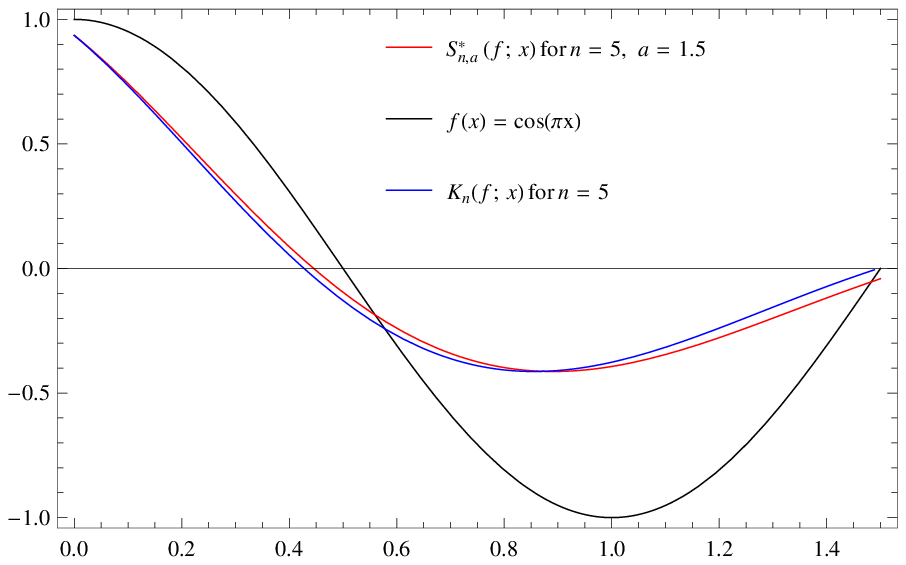}     
      \caption[Description in LOF, taken from~\cite{source}]{The comparison  of the operators $\tilde{S}_{n,a}^*(f;x)$ and $K_n(f;x)$}
    \label{F6}
\end{figure}
\end{example}
By observing to Figure \ref{F6}, it can be found that the better rate of convergence is taking place in the summation-integral-type operators $\tilde{S}_{n,a}^*(f;x)$ rather than  Sz$\acute{\text{a}}$sz-Mirakjan-Kantorovich operators but ultimately they converge to the given function $f(x)=\cos(\pi x)$.

\section{ An extension in the sense of bivariate operators}
To study the approximation properties for the function of two variables, we generalize as an extension to the above operators \ref{NO} into bivariate operators in the space of integral functions to  investigate the rate of convergence with the help of the statistical convergence.
Let $f:C[0,\infty)\times C[0,\infty)\to C[0,\infty)\times C[0,\infty)$, and we define the operators with one parameter as follows:
\begin{eqnarray}\label{NO1}
Y_{m,m,a}^*(f;x,y)=m^2\sum\limits_{k_1=0}^{\infty}\sum\limits_{k_2=0}^{\infty} s_{m,m}^a(x,y)\int\limits_{\frac{k_2}{m}}^{\frac{k_2+1}{m}}\int\limits_{\frac{k_1}{m}}^{\frac{k_1+1}{m}}f(u,v)~du~dv,
\end{eqnarray}
where $s_{m,m}^a(x,y)=a^{\left(\frac{-x-y}{-1+a^{\frac{1}{m}}}\right)}\frac{x^{k_1}y^{k_2}(\log{a})^{k_1+k_2}}{(-1+a^{\frac{1}{m}})^{k_1+k_2}k_1!k_2!}=s_m^a(x)\times s_m^a(y)$.\\

Let us define a function $ e_{i,j}=x^iy^j$, for all $x,y\geq 0$, where $i,j\in \mathbb{N}\cup \{0\}$.
\begin{lemma}\label{L5}
For all $x,~y\geq0$, the bivariate operators \ref{NO1}, satisfy the following equalities: 
\begin{eqnarray*}
&&1.~Y_{m,m,a}^*(e_{00};x,y)=1\\
&&2.~Y_{m,m,a}^*(e_{11};x,y)=\frac{1}{4(-1+a^{\frac{1}{m}})m^2}\Big\{\left(-1+a^{\frac{1}{n}}+2x\log{a} \right)\left(-1+a^{\frac{1}{n}}+2y\log{a} \right) \Big\} \\
&&3.~Y_{m,m,a}^*(e_{22};x,y)=\frac{1}{9(-1+a^{\frac{1}{m}})^4m^4}\Big\{\left( \left(-1+a^{\frac{1}{m}} \right)^2+6(-1+a^{\frac{1}{m}})x\log+3x^2(\log)^2 \right)\\
&&\hspace{3.5cm} \left(\left(-1+a^{\frac{1}{m}} \right)^2+6(-1+a^{\frac{1}{m}})y\log+3y^2(\log)^2 \right) \Big\}\\
&&4.~Y_{m,m,a}^*(e_{33};x,y)=\frac{1}{16\left(-1+a^{\frac{1}{m}} \right)^6m^6}\Big\{\Big(\left(-1+a^{\frac{1}{m}} \right)^3+14\left(-1+a^{\frac{1}{m}} \right)^2x\log{a}+18\left(-1+a^{\frac{1}{m}} \right)x^2(\log{a})^2 \\
&&\hspace{3.5cm}+4x^3(\log{a})^3 \Big) \Big(\left(-1+a^{\frac{1}{m}} \right)^3+14\left(-1+a^{\frac{1}{m}} \right)^2y\log{a}+18\left(-1+a^{\frac{1}{m}} \right)y^2(\log{a})^2 \\ 
&&\hspace{3.5cm}+4y^3(\log{a})^3 \Big) \Big\}.
\end{eqnarray*}
\end{lemma}
\begin{proof}
 To prove (1) of the above Lemma \ref{L5}, we put $f(x,y)=e_{00}=1$ in the above bivariate operators (\ref{NO1}) and we have 
\begin{eqnarray*}
1.~Y_{m,m,a}^*(e_{00};x,y) &=& m^2\sum\limits_{k_1=0}^{\infty}\sum\limits_{k_2=0}^{\infty} s_{m,m}^a(x,y)\int\limits_{\frac{k_2}{m}}^{\frac{k_2+1}{m}}\int\limits_{\frac{k_1}{m}}^{\frac{k_1+1}{m}}~du~dv\\ 
&=& \left(\sum\limits_{k_1=0}^{\infty} s_m^a(x)\right)\left(\sum\limits_{k_2=0}^{\infty}s_m^a(y) \right)\\ 
&=&1.\\
\end{eqnarray*} 
\begin{eqnarray*}
2.~ Y_{m,m,a}^*(e_{11};x,y) &=& m^2\sum\limits_{k_1=0}^{\infty}\sum\limits_{k_2=0}^{\infty} s_{m,m}^a(x,y)\int\limits_{\frac{k_2}{m}}^{\frac{k_2+1}{m}}\int\limits_{\frac{k_1}{m}}^{\frac{k_1+1}{m}}uv~du~dv\\
&=& \frac{m^2}{4} \Big(\sum\limits_{k_1=0}^{\infty}s_m^a(x)\left(\frac{1}{m^2}+\frac{2k_1}{m} \right)  \Big)\Big(\sum\limits_{k_2=0}^{\infty}s_m^a(y)\left(\frac{1}{m^2}+\frac{2k_2}{m} \right) \Big)\\ 
&=& \frac{1}{4(-1+a^{\frac{1}{m}})m^2}\Big\{\left(-1+a^{\frac{1}{m}}+2x\log{a} \right)\left(-1+a^{\frac{1}{m}}+2y\log{a} \right) \Big\} 
\end{eqnarray*}
Similarly, it can be proved for the other equalities. 
\end{proof}
\begin{lemma}
For all $x,y\geq0$ and $m\in\mathbb{N}$, we have
\begin{eqnarray*}
&&1.~Y_{m,m,a}^*((u-x);x,y)=-\frac{(-1+2mx)}{2m}+\frac{x\log{a}}{m(-1+a^{\frac{1}{m}})},\\
&&2.~Y_{m,m,a}^*((v-y);x,y)=-\frac{(-1+2ny)}{2n}+\frac{y\log{a}}{n(-1+a^{\frac{1}{n}})},\\
&&3.~Y_{m,m,a}^*((u-x)^2;x,y)= \frac{(1-3mx+3m^2x^2)}{3m^2}-\frac{2(-1+a^{\frac{1}{m}})(-1+mx)x\log{a}}{\left(-1+a^{\frac{1}{m}}\right)^2m^2}+\frac{x^2(\log{a})^2}{\left(-1+a^{\frac{1}{m}}\right)^2m^2},\\
&&4.~Y_{m,m,a}^*((v-y)^2;x,y)= \frac{(1-3ny+3n^2y^2)}{3n^2}-\frac{2(-1+a^{\frac{1}{n}})(-1+ny)y\log{a}}{\left(-1+a^{\frac{1}{n}}\right)^2n^2}+\frac{y^2(\log{a})^2}{\left(-1+a^{\frac{1}{n}}\right)^2n^2}
\end{eqnarray*}
\end{lemma}
\begin{proof}
One can prove the above all equalities with the help of equalities, which are proved in  \cite{RY1}. So we omit the proof.
\end{proof}
\section{Rate of convergence of bivariate operators}
In this section, we find  rate of convergence of the biivariate operators \ref{NO1}, for function of two variables.

Now, we define the supremum norm, by letting $X=[0,\infty)\times[0,\infty)$, we have
$$\|f\|=\underset{x,y\in X}\sup|f(x,y)|,~~~~f\in C_B(X).$$

Consider the  modulus of continuity $\omega(f;\delta_1,\delta_2)$ for the bivariate operators \ref{NO1}, where $\delta_1,\delta_2>0$ and is defined by:
\begin{eqnarray}
\omega(f;\delta_1,\delta_2)= \{\sup |f(u,v)-f(x,y)|:(u,v),~(x,y)\in X,~\text{and}~|u-x|\leq\delta_1,~|v-y|\leq\delta_2 \}.
\end{eqnarray}
\begin{lemma}\label{L4}
Let $f\in C_B(X)$, then for $\delta_1,~\delta_2>0$,  we have the following properties of modulus of continuity: \\
1.  For given function $f$, $\omega(f;\delta_1,\delta_2)\to 0$ as $\delta_1,\delta_2\to 0$.\\
2. $|f(u,v)-f(x,y)|\leq\omega(f;\delta_1,\delta_2)\left(1+\frac{|u-x|}{\delta_1} \right)\left(1+\frac{|v-y|}{\delta_2} \right)$.\\
Fore more details, see \cite{GAA}.
 \end{lemma}
\begin{theorem}
If $f\in C_B(X)$ and $x, y\in [0,\infty)$, then we have
\begin{eqnarray}
|Y_{m,m,a}^*(f;x,y)-f(x,y)|\leq 4\omega(f;\sqrt{\delta_{m,a}},\sqrt{\delta_{m,a}'}),
\end{eqnarray}
where
\begin{eqnarray}
&& \delta_{m,a}=\Bigg(\frac{(1-3mx+3m^2x^2)}{3m^2}-\frac{2(-1+a^{\frac{1}{m}})(-1+mx)x\log{a}}{\left(-1+a^{\frac{1}{m}}\right)^2m^2}+\frac{x^2(\log{a})^2}{\left(-1+a^{\frac{1}{m}}\right)^2m^2} \Bigg),\label{1}\\
&& \delta_{m,a}'=\Bigg(\frac{(1-3my+3m^2y^2)}{3m^2}-\frac{2(-1+a^{\frac{1}{m}})(-1+my)y\log{a}}{\left(-1+a^{\frac{1}{m}}\right)^2m^2}+\frac{y^2(\log{a})^2}{\left(-1+a^{\frac{1}{m}}\right)^2m^2} \Bigg).\label{2}
\end{eqnarray}
\end{theorem}
\begin{proof}
By using the linearity and the positivity of the defined operators  $Y_{m,m,a}^*(f;x,y)$ (\ref{NO1}) and applying on  (2) of the Lemma (\ref{L4}), then for any $\delta_1,\delta_2>0$, we have
\begin{eqnarray*}
&&|Y_{m,m,a}^*(f;x,y)-f(x,y)|\\
&&\hspace{2cm}\leq Y_{m,m,a}^*(|f(t,s)-f(x,y)|;x,y)\\
&&\hspace{2cm}\leq\omega(f;\delta_1,\delta_2)\left(1+\frac{1}{\delta_1}Y_{m,m,a}^*(|u-x|;x,y) \right)\times\left(1+\frac{1}{\delta_2}Y_{m,m,a}^*(|v-y|;x,y) \right)\\
&&\hspace{2cm}\leq \omega(f;\delta_1,\delta_2)\left(1+\frac{1}{\delta_1}\left( Y_{m,m,a}^*((u-x)^2;x,y)\right)^{\frac{1}{2}} \right)\times\left(1+\frac{1}{\delta_2}\left( Y_{m,m,a}^*((v-y)^2;x,y)\right)^{\frac{1}{2}} \right)\\
&&\hspace{9cm}\text{(using the Cauchy-Schwarz inequality)}.
\end{eqnarray*}
 Next one step will complete the proof. 
\end{proof}
At last, we shall see the rate of convergence of the bivariate operators (\ref{NO1}) in the sense of functions belonging to the Lipschitz class $\text{Lip}_\mathcal{M}(\alpha_1,\alpha_2)$, where $\alpha_1,\alpha_2\in (0,1]$ and $\mathcal{M}\geq0$ is any constant and is defined by:
\begin{eqnarray}
|f(u,v)-f(x,y)|\leq \mathcal{M}|u-x|^{\alpha_1}|v-y|^{\alpha_2},~~~~\forall~x,y,u,v\in [0,\infty).
\end{eqnarray}

Our next approach to prove the theorem for finding the rate of convergence, when the function is belonging to the Lipschitz class.

\begin{theorem}
Let $f\in \text{Lip}_\mathcal{M}(\alpha_1,\alpha_2) $ then for each $f\in C_B(X)$, we have 
\begin{eqnarray*}
|Y_{m,m,a}^*(f;x,y)-f(x,y)|\leq \mathcal{M}\delta_{m,a}^{\frac{\alpha_1}{2}}\delta_{m,a}'^{\frac{\alpha_2}{2}},
\end{eqnarray*}
where $\delta_{m,a}$ and $\delta_{m,a}'$ are defined by (\ref{1}) and (\ref{2}) respectively.
\end{theorem}

\begin{proof}
Since defined bivariate  operators $Y_{m,m,a}^*(f;x,y)$ are linear positive and also $f\in \text{Lip}_\mathcal{M}(\alpha_1,\alpha_2)$, where $\alpha_1,~\alpha_2\in (0,1]$, then we have 
\begin{eqnarray*}
&& |Y_{m,m,a}^*(f;x,y)-f(x,y)|\leq Y_{m,m,a}^*(|f(t,s)-f(x,y)|;x,y)|\\
&&\hspace{3.5cm} \leq Y_{m,m,a}^*(\mathcal{M}|u-x|^{\alpha_1}|v-y|^{\alpha_2};x,y)|\\
&&\hspace{3.5cm} = \mathcal{M} Y_{m,m,a}^*(|u-x|^{\alpha_1};x,y)|\times Y_{m,m,a}^*(|v-y|^{\alpha_2};x,y)|  
\end{eqnarray*}
Applying the H$\ddot{o}$lder inequality with $\mathsf{p'}=\frac{2}{\alpha_1},~\mathsf{q'}=\frac{2}{2-\alpha_1}$ and $\mathsf{p''}=\frac{2}{\alpha_2},~\mathsf{q''}=\frac{2}{2-\alpha_2}$  we have
\begin{eqnarray*}
&& |Y_{m,m,a}^*(f;x,y)-f(x,y)|\leq \mathcal{M} (Y_{m,m,a}^*(u-x)^2;x,y))^{\frac{\alpha_1}{2}}\times(Y_{m,m,a}^*(v-y)^2;x,y))^{\frac{\alpha_2}{2}}\\
&& \hspace{3.5cm}= \mathcal{M}\delta_{m,a}^{\frac{\alpha_1}{2}}\delta_{m,a}'^{\frac{\alpha_2}{2}}.
\end{eqnarray*}
Thus the proof is completed.
\end{proof}

\subsection{Graphical approach of bivariate operators}
Now we shall see that, the convergence of the bivaiate operators  defined by (\ref{NO1}) to the function $f(x,y)$ will be presented by graphical representation.
\begin{example}
Let $f\in C(X)$ and choose $m=5~10,~20$, $a=3$ (fixed),  the convergence of $Y_{m,m,a}^*(f;x,y)$ to the function $f(x,y)$ (blue)  takes place  and is illustrated in Figure \ref{F7}. For the different values of $m$, the corresponding operators $Y_{5,5,a}^*(f;x,y)$, $Y_{10,10,a}^*(f;x,y)$ and $Y_{20,20,a}^*(f;x,y)$ represent red, green and magenta colors respectively. 

\begin{figure}[h!]
    \centering 
  \includegraphics[width=0.82\textwidth]{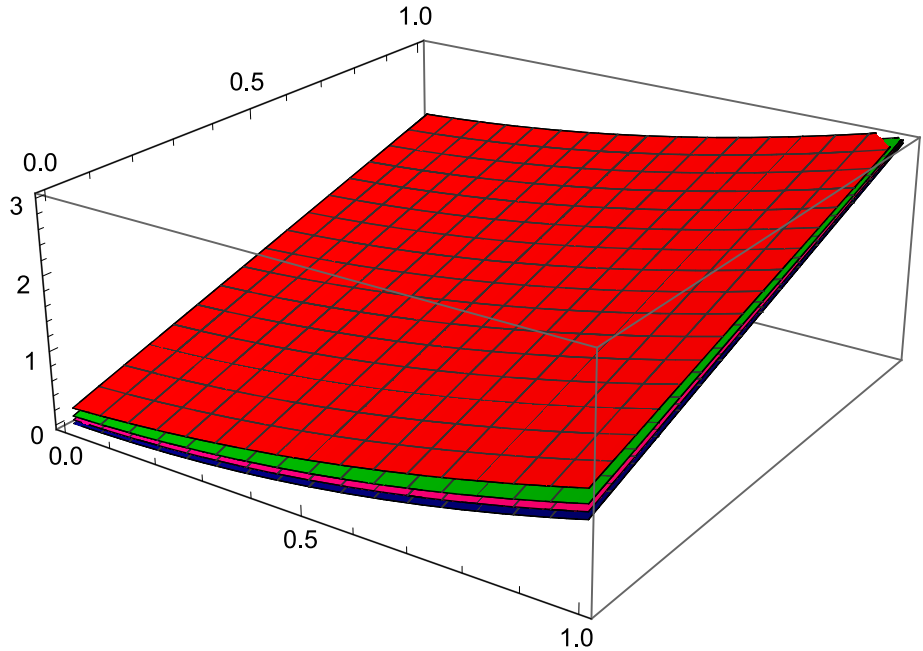}       
      \caption[Description in LOF, taken from~\cite{source}]{Convergence of the operators $Y_{m,m,a}^*(f;x)$ to $f(x,y)$}
    \label{F7}
\end{figure}
\pagebreak
Same here for, same function $f(x,y)$, but for large value of  $m=100,500$ and then corresponding operators are $Y_{100,100,a}^*(f;x,y)$  (red) and $Y_{500,500,a}^*(f;x,y)$ (green)  almost overlap to the function $f(x,y)$ (blue), which is illustrated in Figure \ref{F8}.
\begin{figure}[h!]
    \centering 
  \includegraphics[width=0.82\textwidth]{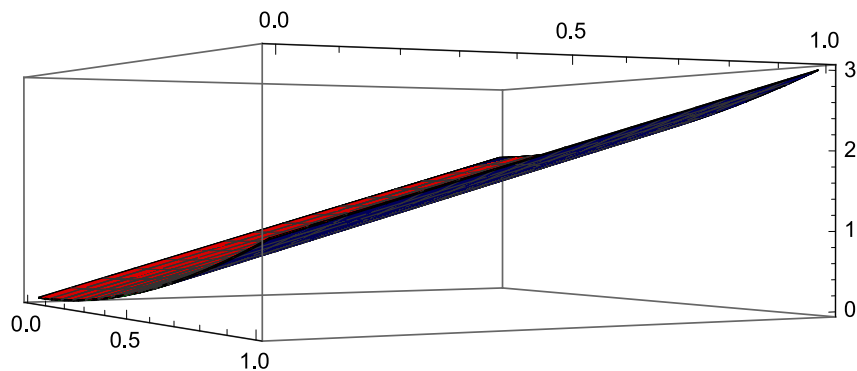}       
      \caption[Description in LOF, taken from~\cite{source}]{Convergence of the operators $Y_{m,m,a}^*(f;x)$ to $f(x,y)$}
    \label{F8}
\end{figure}
\end{example}
\textbf{Concluding remark:} The convergence of the bivariate operators 
$Y_{m,m,a}^*(f;x)$ to the function $f(x,y)$ is taking place as if we increase the value of $m$, i.e. for the large value of $m$, the bivariate operators converge to the function. \\

\textbf{Conclusion:} Convergence of the proposed operators \ref{NO} via statistical sense and order of approximation have been determined, moreover the weighted statistical convergence properties  and the rate of statistical convergence have been investigated in some sense of local approximation results with the help of modulus of continuity. To support the approximation results, the graphical representations took place and along with to stable of the proposed operators (\ref{NO}), a  comparison has been shown and obtained the better rate of convergence. An extension is got for study of the rate of convergence in bivariate sense and graphical analysis has been taken place.

\end{document}